\documentclass[12pt]{article}
\usepackage{graphicx} 
\usepackage{epstopdf}
\usepackage{subfigure}
\usepackage{color}
\usepackage[english]{babel}
\usepackage{amsmath,amsthm}
\usepackage{amssymb}
\usepackage{amsfonts}
\usepackage{booktabs}

\usepackage{graphicx}
\usepackage{subfigure}
\usepackage{tabularx}
\newtheorem{thm}{Theorem}[section]
\newtheorem{defi}{Definition}[section]

\newtheorem{lem}[thm]{Lemma}
\newtheorem{coro}[thm]{Corollary}

\theoremstyle{plain}

\title{The Alon-Tarsi number of Halin graphs}

\author { Zhiguo Li{$^*$}, Qing Ye, Zeling Shao\\
{\small School of Science, Hebei University of Technology, Tianjin 300401, China}
\date{}
\footnote{Corresponding author. E-mail: zhiguolee@hebut.edu.cn}
\footnote{This work was supported in part by the Key Projects of Natural Science Research in Colleges and universities of Hebei Province, China (No.ZD2020130) and the Natural Science Foundation of Hebei Province, China (No. A2021202013). }
}

\begin{document}
\baselineskip 0.65cm

\maketitle

\begin{abstract}
  The \emph{Alon-Tarsi number} of a graph $G$ is the smallest $k$ for which there is an orientation $D$ of $G$ with max outdegree $k-1$ such that the number of Eulerian subgraphs of $G$ with an even number of edges differs from the number of Eulerian subgraphs with an odd number of edges. In this paper, we obtain the Alon-Tarsi number of a Halin graph equals 4 when it is a wheel of even order and 3 otherwise.


\noindent\textbf{2000 MR Subject Classification.} 05C15
\end{abstract}

\section{Introduction}

All graphs considered in this article are finite, and all graphs are either simple graphs or simple directed graphs.
A $k$-$list$ $assignment$ of a graph $G$ is a mapping $L$ which assigns to each vertex $v$ of $G$ a set $L(v)$ of $k$ permissible colors. Given a $k$-list assignment $L$ of $G$, an $L$-$coloring$ of $G$ is a mapping $\phi$ which assigns to each vertex $v$ a color $\phi(v) \in L(v)$ such that $\phi(u) \neq \phi(v)$ for every edge $e = uv$ of $G$.
A graph $G$ is $k$-$choosable$ if $G$ has an $L$-coloring for every $k$-list assignment $L$.
The \emph{choice number} of a graph $G$ is the least positive integer $k$ such that $G$ is $k$-choosable, denoted by $ch(G)$.


In the classic article[1], Alon and Tarsi have obtained the upper bound for the choice number of graphs by applying algebraic techniques, which was later called the Alon-Tarsi number of $G$, and denoted by $AT(G)$ (See e.g. Jensen and Toft (1995) [2]). They have transformed the computation of the Alon-Tarsi number of $G$ from algebraic manipulations to the analysis of the structural properties of $G$. Their characterization can be restated as follows. The \emph{Alon-Tarsi number} of $G$, $AT(G)$, is the smallest $k$ for which there is an orientation $D$ of $G$ with max outdegree $k-1$ such that the number of odd Eulerian subgraphs of $G$ is not the same as the number of even Eulerian subgraphs of $G$.

As pointed out by Hefetz [3], the Alon-Tarsi number has some different features and we are interested in studying $AT(G)$ as an independent graph invariant. Let $\chi_{p}(G)$ be the paint number of $G$ [4]. In [5], U. Schauz generalizes the result of Alon and Tarsi [1]: $ch(G)\leq \chi_{p}(G)\leq AT(G)$ for any graph $G$ and the equalities are not hold in general. Nevertheless, it is also known that the upper bounds of the choice number and the Alon-Tarsi number are the same for several graph classes. For example, In [6], Thomassen has shown that the choice number of any planar graph is at most 5, and it was proved by Schauz in [7] that every planar graph $G$ satisfies $\chi_{p}(G)\leq 5$. As a strengthening of the results of Thomassen and Schauz, X. Zhu proves that every planar graph $G$ has $AT(G)\leq 5$ by Alon-Tarsi polynomial method and $AT$-orientation method [8]. It is of interest to find graph $G$ for which these parameters are equal.

Furthermore, Grytczuk and X. Zhu have used polynomial method to prove that every planar graph $G$ has a matching $M$ such that $AT(G-M)\leq 4$ in [9], it implies a positive answer to the more difficult question $-$ whether every planar graph is 1-defective 4-choosable [10]. Let $T_{m,n}=C_m\Box C_n$ be a toroidal grid, the first author et al. use the same method to show that the Alon-Tarsi number of  $T_{m,n}$ equals $4$  when $m,n$ are both odd and $3$ otherwise in [11]. T. Abe et al. prove that for a $K_5$-minor-free graph $G$, $AT(G)\leq 5$ [12], which generalizes the result of X. Zhu [8].

A $Halin$ graph $H = T \cup C_{n} $ is a plane graph, where $T$ is a tree with no vertex of degree two and at least one vertex of degree three or more, and $C_{n}$ is a cycle connecting the pendant vertices of $T$ in the cyclic order determined by the drawing of $T$. Vertices of $C_{n}$ and $H-C_{n}$ are referred to as $outer$ and $inner$ vertices of $H$, respectively.
In particular, a $wheel$ graph is  a Halin graph which  contains only one inner vertex. In a $wheel$ graph, if we delete an edge of $C_{n}$, the rest of the graph is called a $fan$.


The chromatic number and choice number of Halin graphs are determined in [13] and [14], respectively.
In this article, we obtain the exact values of Alon-Tarsi number of Halin graphs by constructing an $AT$-orientation method.\\

{\bf Main Theorem.}
For a $Halin$ $graph$ $H$, we have
$$AT(H)=\left\{\begin{array}{l}
4,\ \ \ \text{if}~H~\text{is\ a\ wheel\ of\ even\ order};\\
3,\ \ \ \mbox{otherwise}.\\
\end{array} \right.$$


\section{Preliminaries}

\begin{defi} {[1]}
A subdigraph $H$ of a directed graph $D$ is called Eulerian  if $V(H)=V(G)$ and the indegree $d^{-}_{H}(v)$  of every vertex
$v$ of $H$  in $H$ is equal to its outdegree $d^{+}_{H}(v)$. Note that $H$ might not be connected.
 For a digraph $D$, we denote by $\mathcal{E}(D)$ the family of Eulerian subdigraphs of $D$.
 $H$ is even if it has an even number of edges, otherwise, it is odd. Let $\mathcal{E}_{e}(D)$ and $\mathcal{E}_{o}(D)$  denote the
numbers of even and odd Eulerian subgraphs of D, respectively. Let {\rm diff}$(D)=|\mathcal{E}_{e}(D)|-|\mathcal{E}_{o}(D)|$. We say that $D$ is $Alon$-$Tarsi$ if {\rm diff}$(D)\neq0$. If an $orientation$ $D$ of $G$ yields an $Alon$-$Tarsi$ $digraph$, then we say $D$ is an $Alon$-$Tarsi$ $orientation$ $($or an $AT$-$orientation$, for short$)$ of $G$.
\end{defi}

%

Generally, it is difficult to determine whether an orientation $D$ of a graph $G$ is an  AT-orientation.
 Nevertheless, in some cases this problem is very simple.
 Observe that every digraph $D$ has at least one even Eulerian subdigraph, namely, the empty subgraph. If $D$ has no odd directed cycle, then $D$ has no odd Eulerian subdigraph, so $D$ is an $AT$-orientation.

 An\emph{ acyclic orientation} of an undirected graph is an assignment of a direction to each edge (an orientation) that does not form any directed cycle and therefore makes it into a directed acyclic graph.  If $D$ is an acyclic orientation of $G$, then $D$ has no odd Eulerian subdigraph, so $D$ is an $AT$-orientation.
 We denote the maximum outdegree of an acyclic orientation $D$ by $d_{a}$. By the definition of $AT(G)$ we have the following:
 \begin{lem}
 If a graph $G$ has an acyclic orientation $D$ with maximum outdegree $d_a$, then $AT(G)\leq d_a +1$.
 \end{lem}
\begin{defi} {[15]}
 A graph $G$ is $k$-degenerate if there exists an ordering $v_{1}, \ldots, v_{n}$ of vertices of $G$ such that for $i=$ $1, \ldots, n$, the vertex $v_{i}$ has at most $k$ neighbors among $v_{1}, v_{2}, \ldots, v_{i-1}$.
\end{defi}

\begin{lem}
If a graph $G$ is $k$-degenerate, then $AT(G)\leq k+1$.
\end{lem}
\begin{proof}
Suppose that $G$ has degeneracy $k$ and $\sigma$ is a vertex ordering which witnesses this. By orienting each edge toward its
endpoint that appears earlier in  the vertex ordering, we can get an acyclic orientation with maximum outdegree $k$. By Lemma 2.1, $AT(G)\leq k+1.$
\end{proof}
%
%

Let $H=T\cup C_{n}$, where $C_{n}$ is the cycle $v_{1}v_{2}\ldots v_{n}v_{1}$. Then every vertex of $V(C_{n})$ is adjacent to exactly one vertex in $V(H)\setminus V(C_{n})$, and every edge of $E(C_{n})$ is adjacent to exactly two edge in $E(H)\setminus E(C_{n})$. An inner vertex $u$ of a Halin graph $H$ is called $special$ if it is a neighbor of a unique inner vertex. Let $v_{1}, v_{2}, \ldots, v_{k}$ be the neighbors of $u$ on $C_{n}$. If a Halin graph $H$ is not a wheel, then $\{u, v_{1}, v_{2}, \ldots, v_{k}\}$ induces a fan and $H$ contains at least two special inner vertices [13].

\section{Proof of the main theorem}
The proof will be completed by a sequence of lemmas.
\begin{lem} {[14]}
 Every Halin graph is 3-degenerate.
\end{lem}

By Lemma 3.1 and Lemma 2.2, we have
\begin{lem}
 $AT(H)\leq 4$ for each Halin graph $H$.
\end{lem}

\begin{lem}
If $H$ is a wheel of even order, then $AT(H)=4$.
\end{lem}

\begin{proof}
By Lemma 3.2, we know that $AT(H)\leq 4$. $H$ has chromatic number 4, so $AT(H)\geq 4$. Hence $AT(H)= 4$.

%
%
\end{proof}


\begin{figure}[htbp]
\centering
\includegraphics[height=4.1cm, width=0.5\textwidth]{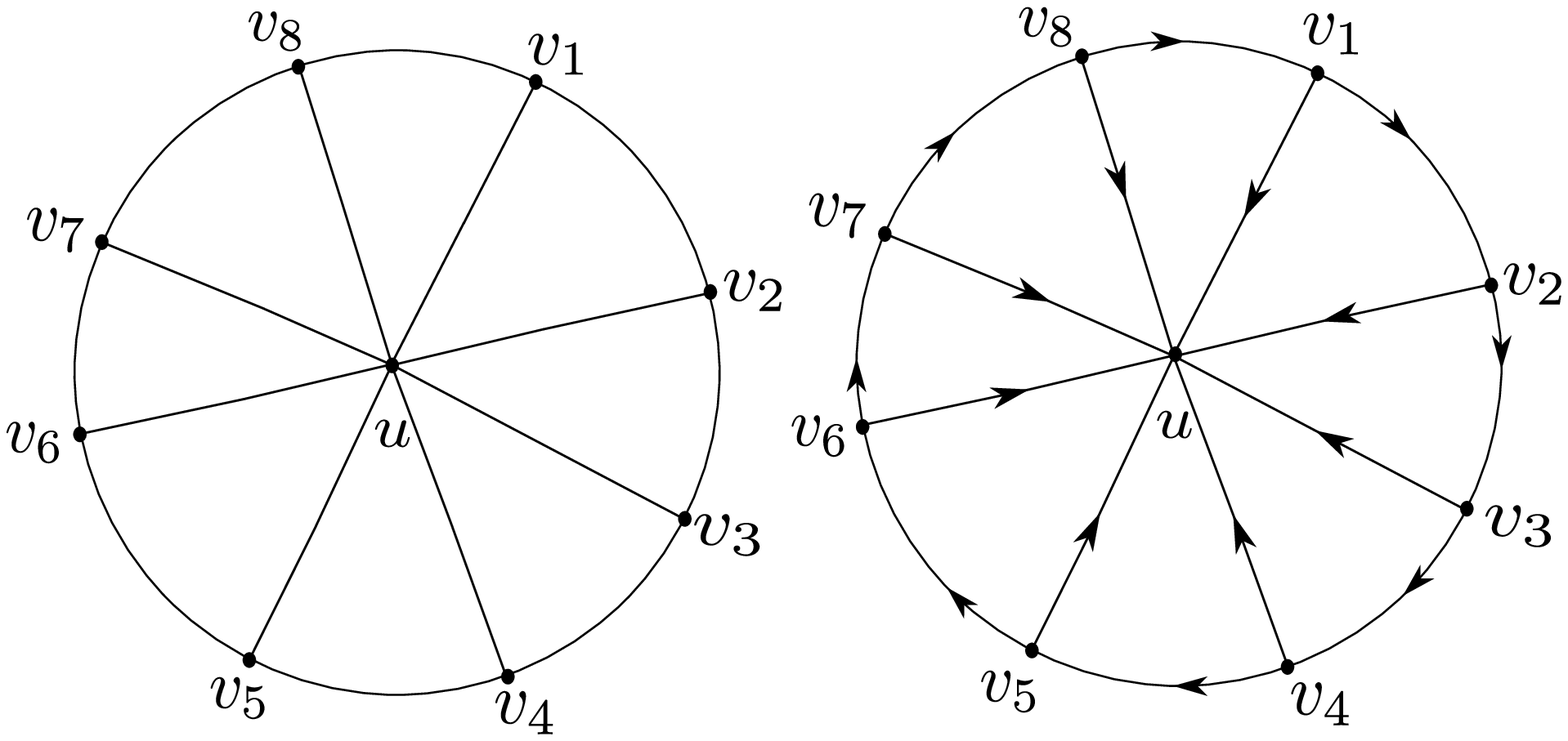}
\centerline{Fig.1  ~The graph $H=C_{8}\cup u$.}
\end{figure}

\begin{lem}
Let $H$ be a Halin graph with an even outer cycle, then $AT(H)=3$.
\end{lem}

\begin{proof}
Assume $H=T\cup C_n$, $n$ is even. We know that $AT(H)\geq 3$  since $\chi (H)=3$. It remains to show that $AT(H)\leq 3$. Consider the following two cases.

{\textbf{Case 1.}}\quad $H$ is a wheel with even outer vertices.


Since $H$ is a wheel, $H$ contains exactly one interior vertex $u$, $d(u)=n$, and $d(v_{i})=3$ for $i=1, 2, \ldots, n$.
Let $D$ be an orientation of $H$ in which the edges of $H$ are oriented in such a way by orientating the outer cycle $C_n$ in clockwise and orientating edge $v_iu$ as $(v_{i}, u), i=1,2,\ldots, n$ [See Figure 1]. $D$ has no odd directed cycle, so $D$ has no odd Eulerian subgraph and hence $D$ is an $AT$-orientation with maximum outdegree 2. Therefore $AT(H)\leq 3$.


{\textbf{Case 2.}}\quad $H$ is not a wheel with even outer vertices.

Similarly, orient $C_{n}$ in clockwise. For $T$, let $X$ be the set of all the inner vertices that are adjacent to $V(C_{n})$. All the arcs between $V(C_{n})$ and $X$ are oriented from $V(C_{n})$ to $X$. The unoriented edges of $T$ induce a subgraph, denoted by $T_{1}$. It is easy to see that $T_{1}$ has at least two leaves.
Let $L_{1}$ and $X_{1}$ be the set of all the leaves of $T_{1}$ and all the vertices of $T_{1}$ that are adjacent to the leaves, respectively.
All the arcs between $L_{1}$ and $X_{1}$ are oriented from $L_{1}$ to $X_{1}$. The unoriented edges of $T_{1}$ induce a subgraph, denoted by $T_{2}$.
Let $L_{2}$ and $X_{2}$ be the set of all the leaves of $T_{2}$ and all the vertices of $T_{2}$ that are adjacent to the leaves, respectively.
All the arcs between $L_{2}$ and $X_{2}$ are oriented from $L_{2}$ to $X_{2}$.

\begin{figure}[htbp]
\centering
\includegraphics[height=4.0cm, width=0.5\textwidth]{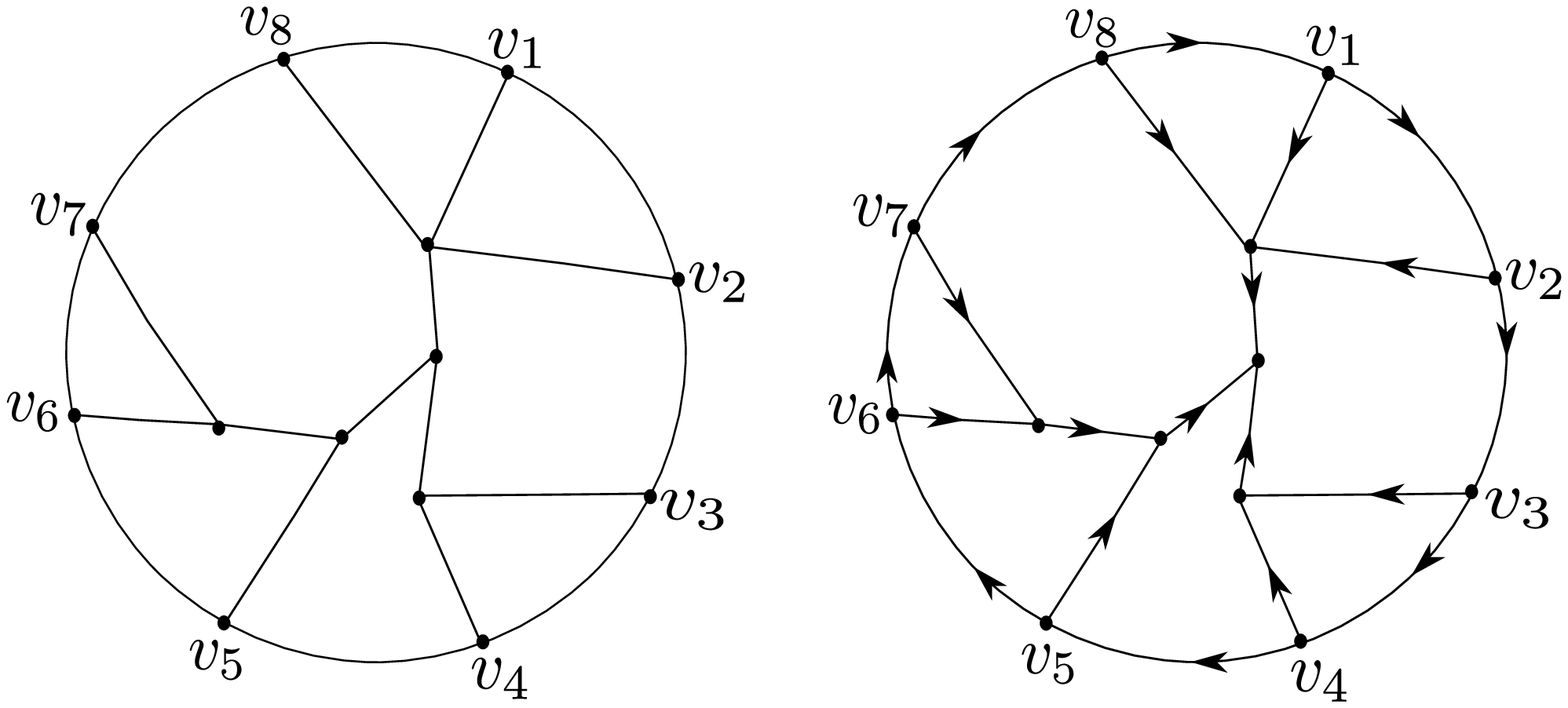}
\centerline{Fig.2  ~$H$ is not a wheel and have even outer vertices.}
\end{figure}

Repeat this process until all edges of $H$ are oriented or only one edge left unoriented, then the edge is oriented arbitrarily. Obviously $D$ has no odd directed cycle, so $D$ has no odd Eulerian subdigraph [See Figure 2]. It is easy to see that the followings hold:
(1) $D$ is an $AT$-orientation,
(2) $d^{+}_{D}(v)\leq1$ for each inner vertex $v$,
(3) $d^{+}_{D}(v)=2$ for each outer vertex $v$.
Hence $AT(H)\leq 3$.  \\
\end{proof}

\begin{lem}{[16]}
Assume that $D$ is a digraph and $V(D)=X_{1} \cup X_{2} .$ For $i=1,2$, let $D_{i}=D\left[X_{i}\right]$ be the subdigraph of $D$ induced by $X_{i} .$ If all the arcs between $X_{1}$ and $X_{2}$ are from $X_{1}$ to $X_{2}$, then $D$ is Alon-Tarsi if and only if $D_{1}, D_{2}$ are both Alon-Tarsi.
\end{lem}



\begin{lem}
Let $H$ be a Halin graph with an  odd outer cycle but not a wheel. Then $AT(H)=3$.
\end{lem}

\begin{proof}
Suppose that $H=T\cup C_{n}$, where $n$ is odd. $H$ has chromatic number 3, so we know that $AT(H)\geq 3$. It remains to show that $AT(H)\leq 3$. Similar to above lemma, we consider two cases.

{\textbf{Case 1.}}\quad $H$ has a special inner vertex $u$ which is adjacent to an odd number of vertices of $C_{n}$.

Let $v_{1}, v_{2}, \ldots, v_{k}$ be the neighbors of $u$ on $C_{n}$ and $k$ is odd.  Then $\{u, v_{1}, v_{2}, \ldots, v_{k}\}$ induces a fan, denoted by  $G_{1}$. Assume $G_{2}=H-V(G_{1})$. The subgraph $G_{1}$ has an orientation $D_1$ as the following way: orient the edge $v_iv_{i+1}$ as $(v_{i},v_{i+1})$, for $1\leq i\leq k-1$, the edge $uv_1$ as $(u, v_{1})$, and the edge $v_{j} u$ as $(v_{j}, u)$ for $j=2, 3, \cdots k$.
Observe that $uv_{1}v_{2}\ldots v_{i}u$ is an odd directed cycle when $i$ is even, and $uv_{1}v_{2}\ldots v_{i}u$ is an even directed cycle when $i$ is odd, for $2\leq i\leq k$. Hence $D_{1}$ contains $k-1$ directed cycles.
Specifically $D_{1}$ contains $\frac{k-1}{2}$ odd directed cycles and $\frac{k-1}{2}$ even directed cycles. It is easy to see that the arc $(u,v_{1})$ is contained in all directed cycles. Since Eulerian subdigraph is the arc disjoint union of directed cycles and empty subdigraph is an even Eulerian subdigraph, $D_{1}$ has $\frac{k-1}{2}$ odd Eulerian subdigraphs and $\frac{k+1}{2}$ even Eulerian subdigraphs. diff$(D_{1})=|\mathcal{E}_{e}(D_{1})|-|\mathcal{E}_{o}(D_{1})|=1 \neq 0$. Therefore $D_{1}$ is an $AT$-orientation of $G_{1}$.

Denote $X=\{v_{k+1},v_{k+2},\ldots, v_{n}\}$, and let $X_{1}$ be the set of all the vertices in $V(G_{2})-X$ that are adjacent to $X$.
$G_{2}$ has an orientation $D_{2}$ that for each $k+1\leq i\leq n-1$, orient edge $v_{i}v_{i+1}$ as $(v_{i}, v_{i+1})$.
All the arcs between $X$ and $X_{1}$ are oriented from $X$ to $X_{1}$. The unoriented edges of $G_{2}$ induce a subgraph, denoted by $T_{1}$.
Let $L_{1}$ and $X_{2}$ be the set of all the leaves of $T_{1}$ and all the vertices of $T_{1}$ that are adjacent to the leaves, respectively.
All the  arcs between $L_{1}$ and $X_{2}$ are oriented from $L_{1}$ to $X_{2}$. The unoriented edges of $T_{1}$ induce a subgraph, denoted by $T_{2}$.
Let $L_{2}$ and $X_{3}$ be the set of all the leaves of $T_{2}$ and all the vertices of $T_{2}$ that are adjacent to the leaves, respectively.
All the arcs between $L_{2}$ and $X_{3}$ are oriented from $L_{2}$ to $X_{3}$.
Repeat this process until all edges of $G_{2}$ are oriented or only one edge left unoriented, then the edge is oriented arbitrarily. It is obvious that $D_{2}$ is an acyclic orientation, hence diff$(D_{2})=|\mathcal{E}_{e}(D_{2})|-|\mathcal{E}_{o}(D_{2})|\neq 0$, $D_{2}$ is an $AT$-orientation of $G_{2}$.

\begin{figure}[htbp]
\centering
\includegraphics[height=4.2cm, width=0.5\textwidth]{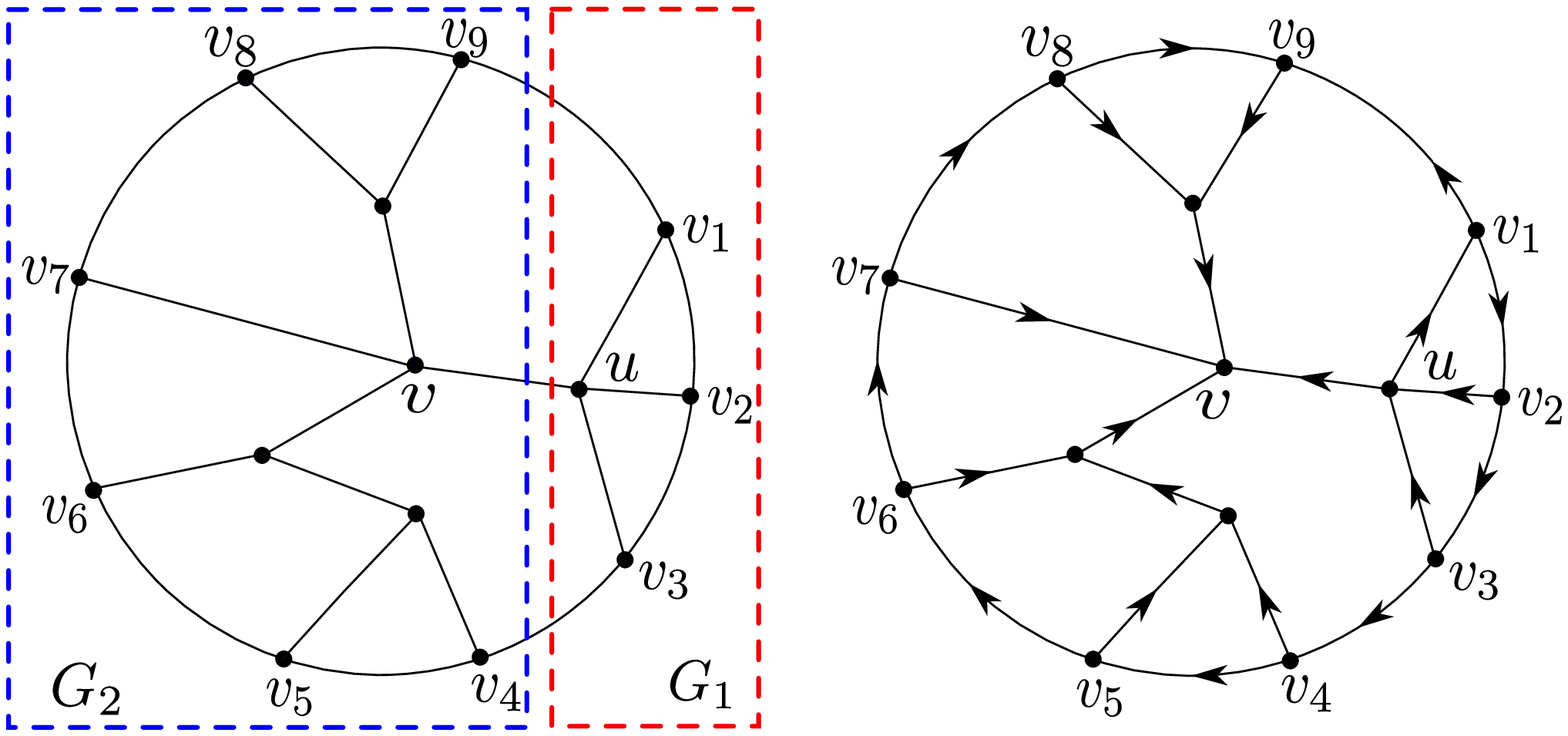}
\centerline{Fig.3  ~The Halin graph $H$ for $n=9$ and $k=3$.}
\end{figure}

Let $v$ be the unique inner vertex which is adjacent to $u$, $u\in V(G_{1})$ and $v\in V(G_{2})$. Let $D$ be obtained from $D_{1}\cup D_{2}$ by adding arcs $(u, v)$, $(v_{1}, v_{n})$ and $(v_{k}, v_{k+1})$. Such that all the arcs between $G_{1}$ and $G_{2}$ are oriented from $G_{1}$ to $G_{2}$ [See Figure 3]. By Lemma 3.5, $D$ is an $AT$-orientation. It is easy to see that $d^{+}_{D}(v)\leq 2$ for every $v\in V(G)$.
Hence $AT(H)\leq 3$.


{\textbf{Case 2.}}\quad All special inner vertices of $H$ are adjacent to an even number of vertices of $C_{n}$.

Let $u$ be a special inner vertex, $v_{1}, v_{2}, \ldots, v_{k}$ be the neighbors of $u$ on $C_{n}$. $\{u, v_{1}, v_{2}, \ldots, v_{k}\}$ induces a fan, denoted by $G_{1}$. Let $G_{2}=H-V(G_{1})$. The subgraph $G_{1}$ has an orientation $D_{1}$ as the following way: orient the edge $v_{i-1}v_{i}$ as $(v_{i}, v_{i-1})$ for $2\leq i\leq k$, and the edge $v_{j}u$ as $(v_{j}, u)$ for $j=1, 2, \ldots k$.
Observed that $D_{1}$ is an acyclic orientation, so $D_{1}$ is an $AT$-orientation of $G_{1}$.

In the tree $T$, any two vertices are connected by exactly one path, so there is a unique $xv_{n}$-path in $T$ connecting $x$ and $v_{n}$, where $x\in V(G_{2})\setminus v_{n}$.
$G_{2}$ has an orientation $D_{2}$ that for $k+2\leq i\leq n$, orient the edge $v_{i-1}v_{i}$ as $(v_{i}, v_{i-1})$, and let every $xv_{n}$-path be a directed path from $x$ to $v_{n}$. It is obvious that all the edges of $G_{2}$ are oriented. Denote $w$ is the unique vertex in $T$ which is adjacent to $v_{n}$, the arc $(w, v_{n})$ is contained in all $xv_{n}$-directed paths. In $G_{2}$, all directed circles are made up of $v_{i}v_{n}$-directed path in $T$ and $v_{n}v_{i}$-directed path in $C_{n}$, for $k+1\leq i\leq n-1$. $D_{2}$ has an even number of directed cycles. Empty subdigraph is an even Eulerian subdigraph of $D_{2}$, hence $|\mathcal{E}(D_{2})|$ is odd. diff$(D_{2})=|\mathcal{E}_{e}(D_{2})|-|\mathcal{E}_{o}(D_{2})|\neq 0$, so $D_{2}$ is an $AT$-orientation of $G_{2}$.

\begin{figure}[htbp]
\centering
\includegraphics[height=4.4cm, width=0.5\textwidth]{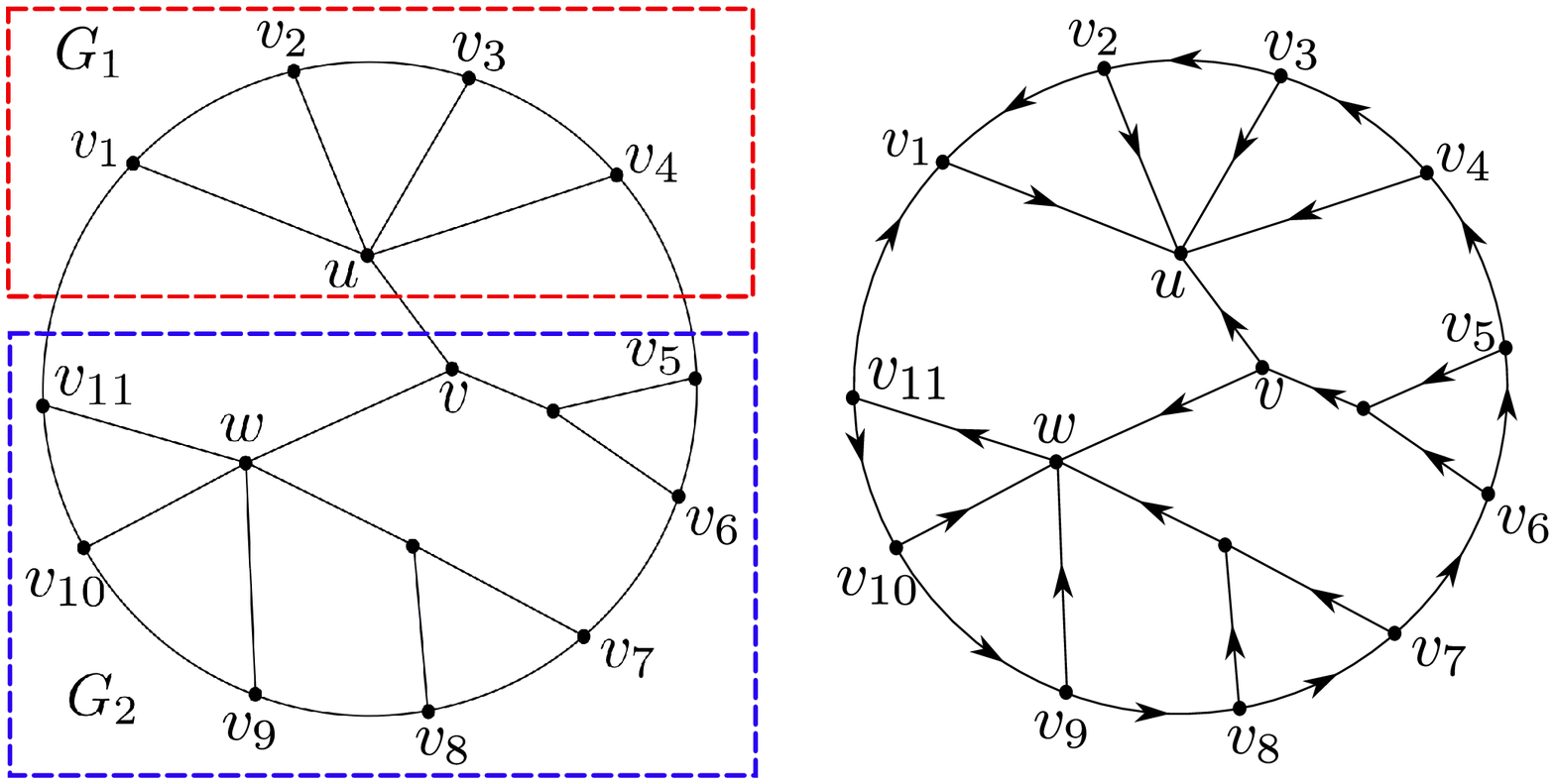}
\centerline{Fig.4  ~The Halin graph $H$ for $n=11$ and $k=4$.}
\end{figure}

Let $v$ be the unique inner vertex which is adjacent to $u$, $u\in V(G_{1})$ and $v\in V(G_{2})$. Let $D$ be obtained from $D_{1}\cup D_{2}$ by adding arcs $(v, u)$, $(v_{n}, v_{1})$ and $(v_{k+1}, v_{k})$. Obviously all the arcs between $G_{1}$ and $G_{2}$ are oriented from $G_{2}$ to $G_{1}$ [See Figure 4]. In a similar way as case 1, we can get  $AT(H)\leq 3$.
%

\end{proof}

\begin{coro}
For a $Halin$ $graph$ $H$, we have
$$\chi(H)=ch(H)=\chi_{P}(H)=AT(H)=\left\{\begin{array}{l}
4,\ \mbox{\rm if}\ H\ {\rm is\ a\ wheel\ of\ even\ order};\\
3,\ \mbox{\rm otherwise}.\\
\end{array} \right.$$
\end{coro}




\end{document}